\documentclass[11pt,fleqn]{article}
\usepackage{amssymb}
\usepackage{amsfonts}
\usepackage{amsmath}
\usepackage{amsthm}
\usepackage{graphicx}
\usepackage{epsfig}
\usepackage{psfrag}
\usepackage{color}
\usepackage{dsfont}
\makeatletter
\xdef\@endgadget#1{{\unskip\nobreak\hfil\penalty50\hskip1em\hbox{}\nobreak
    \hfil#1\parfillskip=0pt\finalhyphendemerits=0\par}}
\def\@qedsymbol{${}_\blacksquare$}
\def\qed{\@endgadget{\@qedsymbol}}
\newtheorem{lemma}{Lemma}[section]
\newtheorem{theorem}[lemma]{Theorem}

\newtheorem{example}[lemma]{Example}
\newtheorem{definition}[lemma]{Definition}

\newtheorem{proposition}[lemma]{Proposition}
\newtheorem{remark}[lemma]{Remark}
\newcommand{\mR}{\mathbb{R}}

\newcommand{\X}{\mathcal{X}}
\newcommand{\cS}{\mathcal{S}}

\newcommand{\bq}{\begin{equation}}
\newcommand{\eq}{\end{equation}}

\def\BibTeX{{\rm B\kern-.05em{\sc i\kern-.025em b}\kern-.08em
    T\kern-.1667em\lower.7ex\hbox{E}\kern-.125emX}}

\title{\LARGE \bf Cyclo-dissipativity revisited}

\author{A.J. van der Schaft
\thanks{A.J. van der Schaft is with the Bernoulli Institute for Mathematics, Computer
Science and AI, University of Groningen, PO Box 407, 9700 AK, the
Netherlands
        {\tt\small a.j.van.der.schaft@rug.nl}}
}

\date{}

\begin{document}

\maketitle
\thispagestyle{empty}
\pagestyle{empty}

\begin{abstract}
Starting from a symmetrization and extension of the basic definitions and results of dissipativity theory we obtain new results on cyclo-dissipativity; in particular their external characterization and description of the set of storage functions.
\end{abstract}

\section{Introduction}
Dissipativity theory originates from the seminal paper \cite{willems1972}. It unifies classical theory, centered around the passivity and small-gain theorems, with Lyapunov function theory for autonomous dynamical systems. In particular, it aims at deriving Lyapunov functions for large-scale interconnected systems, based on the knowledge of the component systems, and the way they are coupled to each other. Furthermore, it directly relates to physical systems theory, network synthesis, and optimal control. 

The more general notion of {\it cyclo-dissipativity} was first formulated\footnote{I would like to thank Pablo Borja and Romeo Ortega for pointing out to me the somewhat forgotten references \cite{hillmoylan1975} and \cite{willems1973}.} in \cite{willems1973}, aimed at extending the stability analysis based on dissipativity towards instability theorems. Implicitly the notion was already present in \cite{willems1971}, motivated by infinite-horizon optimal control. In the technical report \cite{hillmoylan1975} cyclo-dissipativity was further explored, extending the fundamental results obtained in \cite{willems1972} for ordinary dissipativity. Since then the notion of cyclo-dissipativity has not received much detailed attention, although it regularly appears in passivity-based control, e.g., \cite{ortega}. 

In the current note we will revisit the notion of dissipativity, by unifying earlier definitions and developments. This will turn out to be instrumental for developing a more complete theory of cyclo-dissipativity, extending the results of \cite{hillmoylan1975}. Finally, the developed theory will be illustrated on the formulation of the Clausius inequality in thermodynamics, and on the recently introduced notion of one-port cyclo-passivity.

\section{Dissipativity revisited}
In this section we recall the basic definitions and results of dissipativity theory as developed in the groundbreaking paper \cite{willems1972} (with some extensions due to \cite{hillmoylan1980} and \cite{passivitybook}) and put them into a more general and unifying context, as a preparation for the results on cyclo-dissipativity in the next section.

Consider a nonlinear input-state-output system
\bq
\Sigma:  \; 
\begin{array}{rcl}
\dot{x} & = & f(x,u), \quad x \in \X, u \in \mR^m \\[2mm]
y & = & h(x,u), \quad y \in \mR^p
\end{array}
\eq
on an $n$-dimensional state space manifold $\X$. Consider a {\it supply rate}
\bq
s:   \mR^m \times \mR^p \to \mR
\eq
Throughout it will be assumed that for all solutions of $\Sigma$ the integrals $\int_{t_1}^{t_2} s(u(t),y(t)) dt$ are well-defined for all $t_1,t_2$.

Given the system $\Sigma$ and supply rate $s$ the, possibly extended, function
$S: \X \to -\infty \cup \mR \cup \infty$ satisfies the {\it dissipation inequality} if\footnote{Assuming differentiability the dissipation inequality \eqref{DIE} is easily seen \cite{passivitybook} to be equivalent to the {\it differential dissipation inequality} $\dot{S}(x,u) \leq s(u,h(x,u))$, where $\dot{S}(x,u):=\frac{\partial S}{\partial x^T}(x)f(x,u)$.}
\bq
\label{DIE}
S(x(t_2)) \leq S(x(t_1)) + \int_{t_1}^{t_2} s(u(t),y(t)) dt
\eq
holds for all $t_1 \leq t_2$, all input functions $u:[t_1,t_2] \to \mR^m$, and all initial conditions $x(t_1)$, where $y(t)=h(x(t),u(t))$, with $x(t)$ denoting the solution of $\dot{x}=f(x,u)$ for initial condition $x(t_1)$ and input function $u:[t_1,t_2] \to \mR^m$. In particular this implies that if $S(x(t_2))$ equals $\infty$, then so does $S(x(t_1))$, and if $S(x(t_1))$ equals $- \infty$, then so does $S(x(t_2))$.

A {\it non-extended} function $S : \X \to \mR$ satisfying the dissipation inequality \eqref{DIE} is called a {\it storage function}\footnote{Note that we do not yet require $S$ to be nonnegative or bounded from below.}. This leads to the following standard definition of dissipativity as pioneered in the seminal paper \cite{willems1972}; see also \cite{hillmoylan1980}.
\begin{definition}
The system $\Sigma$ is {\it dissipative} with respect to the supply rate $s$ if there exists a {\it nonnegative} storage function $S: \X \to \mR^+$. If the nonnegative storage function $S: \X \to \mR^+$ satisfies \eqref{DIE} with equality then the system is called {\it lossless}.
\end{definition}
In case of the supply rate $s(u,y)=y^Tu, \, u,y\in \mR^m$ 'dissipativity' is usually referred to as '{\it passivity}'.

In order to characterize dissipativity (and subsequently the weaker property of {\it cyclo-dissipativity}), let us define the following, possibly extended, functions\footnote{Here $a$ refers to 'available storage', and $r$ to 'required storage'. Note that we deviate from the standard notation, where $S_r$ in fact refers to the function $S_{rc}$ as defined next. In fact, in the standard treatments of dissipativity \cite{willems1972, willems1973, hillmoylan1975, hillmoylan1980, passivitybook} only the functions $S_a$ and $S_{rc}$ (there denoted as $S_r$) are used. The present set-up aims at 'symmetrizing' the picture; also with a view on cyclo-dissipativity.} $S_a: \X \to \mR \cup \infty$ and $S_r: \X \to - \infty \cup \mR $
\bq
\begin{array}{l}
S_a(x) = \sup_{u,T\geq 0 \mid x(0)=x} - \int_0^Ts(u(t),y(t)) dt \\[2mm]
S_r(x) = \inf_{u,T\geq 0 \mid x(0)=x}  \int_{-T}^0s(u(t),y(t)) dt
\end{array}
\eq
Interpreting the supply rate $s$ as the 'power' supplied to the system, and the storage function $S$ as the 'energy' stored in the system, the function $S_a(x)$ equals the maximally extractable 'energy' from the system at state $x$ \cite{willems1972}. Similarly, $S_r$ equals the 'energy' that needs to be minimally supplied to the system while bringing it to state $x$.
Obviously $S_a,S_r$ satisfy
\bq
S_a(x) \geq 0, \quad S_r(x) \leq 0
\eq
Furthermore, assuming {\it reachability} from a certain ground-state $x^*$ and {\it controllability} to this same state $x^*$, we define the, possibly extended, functions $S_{ac}: \X \to \mR \cup \infty, S_{rc} : \X \to - \infty \cup \mR$ as\footnote{Here $c$ stands for 'constrained', since either $x(T)=x^*$ or $x(-T)=x^*$.}
\bq
\begin{array}{l}
S_{ac}(x) = \sup_{u,T\geq 0 \mid x(0)=x,x(T)=x^*} - \int_0^Ts(u(t),y(t)) dt \\[2mm]
S_{rc}(x) = \inf_{u,T\geq 0 \mid x(-T)=x^*,x(0)=x}  \int_{-T}^0s(u(t),y(t)) dt,
\end{array}
\eq
which again have an obvious interpretation in terms of 'energy'.

Clearly for all $x \in \X$
\bq
\label{ineq}
\begin{array}{l}
- \infty < S_{ac}(x) \leq S_a(x) \\[2mm]
S_r(x) \leq S_{rc}(x) < \infty
\end{array}
\eq
Furthermore, it is straightforward\footnote{The first equality in \eqref{EQ} already figures in \cite{passivitybook}; the second one is similar.} to check that the above four functions are related by
\bq
\label{eq}
\begin{array}{l}
S_a(x^*) = \sup_x - S_{rc}(x) \; (= \\[2mm]
\sup_x \sup_{u,T\geq 0 \mid x(-T)=x^*,x(0)=x}  - \int_{-T}^0 s(u(t),y(t)) dt  \, ) \\[6mm]
S_r(x^*) = \inf_x -S_{ac}(x) \;(= \\[2mm]
\inf_x \inf_{u,T\geq 0 \mid x(0)=x,x(T)=x^*} \int_0^Ts(u(t),y(t)) dt \, ) 
\end{array}
\eq
In particular it follows that
\bq
\label{rel}
\begin{array}{l}
S_a(x^*) < \infty \Leftrightarrow \inf_x S_{rc}(x) > - \infty \\[4mm]
S_r(x^*) > - \infty \Leftrightarrow \sup_x S_{ac}(x) < \infty 
\end{array}
\eq
By using the definitions of infimum and supremum it is easily verified (see \cite{willems1971,willems1972,passivitybook}) that all four (possibly extended) functions $S_a,S_r, S_{ac},S_{rc}$ satisfy the dissipation inequality \eqref{DIE}.

The following theorem summarizes some of the main findings of dissipativity theory as formulated in \cite{passivitybook}, extending the fundamental results of \cite{willems1972}; see also \cite{hillmoylan1980}. 
\begin{theorem}
$\Sigma$ is dissipative if and only if $S_a(x) < \infty$ for all $x \in \X$ (that is, $S_a: \X \to \mR$). If $\Sigma$ is reachable from $x^*$ then $\Sigma$ is dissipative if and only if $S_a(x^*) < \infty$, or equivalently (see \eqref{rel}) $\inf_x S_{rc}(x) > - \infty $. Furthermore, if $\Sigma$ is dissipative then $S_a$ is a nonnegative storage function satisfying $\inf_x S_a(x)=0$, and all other nonnegative storage functions $S$ satisfy
\bq
S_a(x) \leq S(x) - \inf_x S(x),\quad x \in \X
\eq
Moreover, if $\Sigma$ is dissipative then $S_{rc} - \inf_xS_{rc}(x)$ is a nonnegative storage function, and all other nonnegative storage functions $S$ satisfy
\bq
S(x) - S(x^*) \leq S_{rc}(x) 
\eq
\end{theorem}
\begin{remark}
In case the system is reachable from $x^*$ while furthermore $S_a(x^*)=0$ (maximally extractable energy from ground state $x^*$ is zero), then it directly follows from the dissipation inequality \eqref{DIE} that dissipativity is equivalent to the {\it external characterization}
\bq
\int_0^Ts(u(t),y(t)) dt \geq 0,
\eq
for all $u:[0,T] \to \mR^m, T\geq 0$, where $y(t)=h(x(t),u(t))$ with $x(t)$ the solution of $\dot{x}=f(x,u)$ for initial condition $x(0)=x^*$.
This is sometimes taken as the definition of dissipativity; especially in the linear case with $x^*=0$. Note however that in a nonlinear context there is often no natural groundstate.
\end{remark}
In the next proposition similar results will be derived for the newly defined functions $S_r$ and $S_{ac}$; however with the key difference that $S_r \leq 0$, and thus does {\it not} correspond to dissipativity.
\begin{proposition}
If $S_r(x)>-\infty$ for all $x \in \X$ then $S_r$ is a nonpositive storage function. If $\Sigma$ is controllable to $x^*$ then $S_r(x)>-\infty$ for all $x \in \X$ if and only if $S_r(x^*) >- \infty$, or equivalently $\sup_x S_{ac}(x) < \infty $. Furthermore, if $S_r(x)>-\infty$ for all $x \in \X$ then $S_r$ is a nonpositive storage function satisfying $\sup_x S_r(x)=0$, and all other nonpositive storage functions $S$ satisfy
\bq
S_r(x) \geq S(x) - \sup_x S(x)
\eq
Moreover, if $S_r(x^*) >- \infty$ then $S_{ac} - \sup_xS_{ac}(x)$ is a nonpositive storage function, and all nonpositive storage functions $S$ satisfy
\bq
S(x) - S(x^*) \geq S_{ac}(x) 
\eq
\end{proposition}
\begin{remark}
The {\it reversed} dissipation inequality 
\bq
\label{DIErev}
S(x(t_2)) \geq S(x(t_1)) + \int_{t_1}^{t_2} s(u(t),y(t)) dt, 
\eq
appears in optimal control theory, with $s$ the running cost; see e.g. \cite{willems1971, passivitybook}. Obviously the reversed dissipation inequality is obtained from \eqref{DIE} by replacing $s$ by $-s$ and $S$ by $-S$ (and thus $-SR\geq0$).
\end{remark}
\begin{remark}
Under certain conditions on the supply rate $s$ and the system $\Sigma$ the definitions of $S_a$ and $S_r$ can be modified to the {\it indefinite} integrals
\bq
\begin{array}{l}
S_a(x) = \sup_{u \mid x(0)=x} - \int_0^{\infty}s(u(t),y(t)) dt \\[2mm]
S_r(x) = \inf_{u \mid x(0)=x}  \int_{-\infty}^0s(u(t),y(t)) dt
\end{array}
\eq
with $x(t)$ approaching $x^*$ for $t \to \infty$ in the first case (see \cite{willems1971}, \cite{passivitybook} for precise statements), and for $t \to -\infty$ in the second case. This is the context of the developments in \cite{willems1971}, where however the system dynamics was assumed to be linear.
\end{remark}
\begin{example}
Consider the scalar system
\bq
\Sigma: \, \dot{x}=u, \; y=e^x
\eq
with passivity supply rate $s(u,y)=uy$.
Then
\bq
\begin{array}{l}
S_a(x)= \sup_{T\geq 0} - (e^{x(T)} - e^x) = e^x \\[2mm]
S_r(x) = \inf_{T\geq 0} (e^x - e^{x(-T)}) = - \infty
\end{array}
\eq
Hence $\Sigma$ is passive with nonnegative storage function $e^x$ (unique up to a constant). Obviously, $\Sigma$ is reachable from and controllable to $x^*=0$ (for example). With respect to $x^*=0$
\bq
S_{ac}(x)= - (e^0-e^x)= e^x -1, \; S_{rc}(x)= -(e^x-e^0)= e^x -1,
\eq
i.e., $S_{ac}=S_{rc} = e^x-1$.
%Furthermore
%\bq
%S_a(0) = \sup_x - S_{rc}(x) = \sup_x -(e^x-1) = 1
%\eq
%All nonnegative storage functions $S$ satisfy $S_a(x) \leq S(x) - \inf S(x)$, and $S(x) \leq S_{rc}(x) + S(0)= (e^x - 1) + 1 =S_a(x)$ (since the system is lossless). 
%
%$\Sigma$ is of course also cyclo-dissipative. 
Any storage function $S$ satisfies $S_{ac}(x) \leq S(x) -S(0) \leq S_{rc}(x)$, i.e., $S(x) -S(0)= e^x - 1$. Hence the storage function is unique up to a constant.\\
Note that the example could be extended to e.g. the passive (non-lossless) case $\dot{x}= - x^2 + u, y=e^x$.
\end{example}

\section{Cyclo-dissipativity}
Next we come to the study of {\it cyclo-dissipativity}, as coined in \cite{willems1973} with instability theorems for interconnected systems in mind, and further developed in \cite{hillmoylan1975}.
\begin{definition}
$\Sigma$ is {\it cyclo-dissipative} if 
\bq
\label{cyclic}
\oint s(u(t),y(t)) dt \geq 0
\eq
for all $T\geq 0$ and all $u:[0,T] \to \mR^m$ such that $x(T)=x(0)$.
Furthermore, $\Sigma$ is called cyclo-dissipative {\it with respect to} $x^*$ if 
\bq
\label{cyclicstar}
\oint s(u(t),y(t)) dt \geq 0
\eq
for all $T\geq 0$ and all $u:[0,T] \to \mR^m$ such that $x(T)=x(0)=x^*$. In case \eqref{cyclic} or \eqref{cyclicstar} holds with equality, we speak about cyclo-losslessness.
\end{definition}

The following proposition is obvious, and follows from substituting $x(T)=x(0)$ in \eqref{DIE}.
\begin{proposition}
\label{obvious}
If there exists a storage function for the system $\Sigma$ then $\Sigma$ is cyclo-dissipative.
\end{proposition}
The following theorem extends the results in \cite{hillmoylan1975} in a number of directions.
\begin{theorem}
\label{cycthm}
Assume that $\Sigma$ is reachable from $x^*$ and controllable to $x^*$. Then $\Sigma$ is cyclo-dissipative with respect to $x^*$ if and only if
\bq
\label{acrc}
S_{ac}(x) \leq S_{rc}(x), \quad x \in \X
\eq
In particular, if $\Sigma$ is cyclo-dissipative with respect to $x^*$ then both $S_{ac}$ and $S_{rc}$ are storage functions, and thus $\Sigma$ is cyclo-dissipative.
Furthermore, if $\Sigma$ is cyclo-dissipative with respect to $x^*$ then
\bq
S_{ac}(x^*) = S_{rc}(x^*)=0,
\eq
and any other storage function $S$ satisfies
\bq
\label{**}
S_{ac}(x) \leq S(x) - S(x^*) \leq S_{rc}(x) 
\eq
\end{theorem}
\begin{proof}
Suppose $\Sigma$ is cyclo-dissipative with respect to $x^*$. Since $\Sigma$ is reachable from and controllable to $x^*$, there exist trajectories $x(t)$ with $x(-T)=x(T)=x^*, x(0)=x$. For any of these 
\bq
\int_{-T}^0 s(u(t),y(t)) dt + \int_{0}^T s(u(t),y(t)) dt \geq 0,
\eq
and thus 
\bq
\int_{-T}^0 s(u(t),y(t)) dt \geq - \int_{0}^T s(u(t),y(t)) dt
\eq
Taking infimum on the left-hand side and supremum on the right-hand side we obtain $S_{rc}(x) \geq S_{ac}(x)$ for any $x$.
Conversely, let \eqref{acrc} hold. Since by definition $S_{ac}(x)> -\infty, S_{rc}(x) < \infty$, it follows that both $S_{ac}$ and $S_{rc}$ are storage functions. Thus by Proposition \ref{obvious} $\Sigma$ is cyclo-dissipative.\\
Furthermore, let $\Sigma$ be cyclo-dissipative with respect to $x^*$. By definition $S_{ac}(x^*) \geq 0$. On the other hand, in view of \eqref{cyclicstar}, we also have $S_{ac}(x^*) \leq 0$, thus implying $S_{ac}(x^*) = 0$. Similarly, $S_{rc}(x^*) \leq 0$, and in view of \eqref{cyclicstar} also $S_{rc}(x^*) \geq 0$, thus implying $S_{rc}(x^*) = 0$.\\
Finally, let $\Sigma$ be cyclo-dissipative with respect to $x^*$. Then
\bq
S(x^*) - S(x) \leq \int_0^T s(u(t),y(t)) dt
\eq
for any trajectory from $x$ at $t=0$ to $x^*$ at $t=T$. This is equivalent to
\bq
\begin{array}{c}
S(x) - S(x^*) \geq \\[2mm]
\sup_{u,T\geq 0 \mid x(0)=x,x(T)=x^*} - \int_0^T s(u(t),y(t)) dt =S_{ac}(x)
\end{array}
\eq
Similarly 
\bq
S(x) - S(x^*) \leq \int_{-T}^0 s(u(t),y(t)) dt
\eq
for any trajectory from $x^*$ at $t=-T$ to $x$ at $t=0$. This is equivalent to
\bq
\begin{array}{c}
S(x) - S(x^*) \leq \\[2mm]
\inf_{u,T\geq 0 \mid x(-T)=x^*,x(0)=x}  \int_{-T}^0 s(u(t),y(t)) dt =S_{rc}(x)
\end{array}
\eq
\end{proof}
\begin{remark}
Note that all results concerning interconnection of dissipative systems as developed in \cite{willems1972}, \cite{hillmoylan1980}, \cite{passivitybook} remain to hold for cyclo-dissipative systems, with the difference that the Lyapunov function obtained for the interconnected system by summing the storage functions for the cyclo-dissipative system components is no longer nonnegative. Hence in principle only {\it instability} results can be inferred; this is the motivation for cyclo-dissipativity in \cite{willems1973}.
\end{remark}
\begin{remark}
Cyclo-dissipativity, instead of dissipativity, is not uncommon in physical systems modeling; especially in the nonlinear case. For example, the gravitational energy between two masses is proportional to $-\frac{1}{r}$, with $r\geq 0$ the distance between the two masses. This function is {\it not} bounded from below, and thus cannot be turned into a nonnegative storage function by addition of a constant. Energy functions that are not bounded from below also appear frequently in thermodynamic systems; see Example \ref{therm} later on.
\end{remark}
Finally, the functions $S_{ac}$ and $S_{rc}$ depend on the choice of the ground-state $x^*$, and one may wonder about the relation between these functions for different ground-states. Partial information is provided in the following proposition.
\begin{proposition}
Denote $S_{ac}$ for ground-state $x^*$ by $S^*_{ac}$, and for ground-state $x^{**}$ by $S^{**}_{ac}$. Similarly define $S^{*}_{rc}$ and $S^{**}_{rc}$. Then
\bq
S^{*}_{ac}(x^{**}) + S^{**}_{ac}(x^{*}) \leq 0, \quad S^{*}_{rc}(x^{**}) + S^{**}_{rc}(x^{*}) \geq 0
\eq
\end{proposition}
\begin{proof}
By the first inequality of \eqref{**}
\bq
S^{*}_{ac}(x) \leq S^{**}_{ac}(x) - S^{**}_{ac}(x^{*}), \quad S^{**}_{ac}(x) \leq S^{*}_{ac}(x) - S^{*}_{ac}(x^{**})
\eq
This implies
\[
S^{*}_{ac}(x) + S^{**}_{ac}(x^{*}) \leq S^{*}_{ac}(x) - S^{*}_{ac}(x^{**})
\]
and thus $S^{*}_{ac}(x^{**}) + S^{**}_{ac}(x^{*}) \leq 0$. The second inequality follows analogously from the second inequality of \eqref{**}.
\end{proof}

\subsection{The non-controllable and/or non-reachable case}
The definitions of $S_{ac}$ and $S_{rc}$ can be extended to the non-controllable/non-reachable case by defining (see already \cite{hillmoylan1975})
\bq
\begin{array}{rcll}
S_{ac}(x) & = & - \infty, \quad & \mbox{if $x$ is not controllable to $x^*$} \\[2mm]
S_{rc}(x) & = & \infty, \quad &\mbox{if $x$ is not reachable from $x^*$}
\end{array}
\eq
Then it is easily verified that these extended functions $S_{ac}: \X \to -\infty \cup \mR$ and $S_{rc}:\X \to \mR \cup \infty$ still satisfy the dissipation inequality \eqref{DIE}, while also \eqref{rel} remains to hold.  On the other hand, \eqref{ineq} needs to be replaced by
\bq
\label{ineq1}
- \infty \leq S_{ac}(x) \leq S_a(x), \quad S_r(x) \leq S_{rc}(x) \leq \infty, \quad x \in \X
\eq
With regard to the characterization of cyclo-dissipativity we note that trivially $S_{ac}(x) \leq S_{rc}(x)$ whenever $x$ is non-controllable to $x^*$ or non-reachable from $x^*$. Thus also in the non-controllable/non-reachable case $\Sigma$ is cyclo-dissipative with respect to $x^*$ if and only if \eqref{acrc} holds. On the other hand, if $\Sigma$ is not controllable and not reachable then $S_{ac}$ and $S_{rc}$ are extended functions, and hence no cyclo-dissipativity can be concluded. We arrive at the following extension of Theorem \ref{cycthm} with completely analogous proof.
\begin{theorem}
$\Sigma$ is cyclo-dissipative with respect to $x^*$ if and only if
\bq
\label{acrc1}
S_{ac}(x) \leq S_{rc}(x), \quad x \in \X
\eq
Furthermore, if $\Sigma$ is cyclo-dissipative with respect to $x^*$ then
\bq
S_{ac}(x^*) = S_{rc}(x^*)=0,
\eq
and any other (possibly extended) function $S$ satisfying the dissipation inequality \eqref{DIE} is such that
\bq
S_{ac}(x) \leq S(x) - S(x^*) \leq S_{rc}(x), \quad x \in \X
\eq
%$\Sigma$ is cyclo-dissipative if and only if $\Sigma$ is cyclo-dissipative with respect to all $x^*$.
\end{theorem}
Here it could also be noted that, similar to dissipativity theory \cite{willems1972}, the set of storage functions is always {\it convex}.

\subsection{Clausius' inequality}
An interesting application of cyclo-dissipativity concerns the formulation and implications of the Second law of thermodynamics. The standard argumentation in classical thermodynamics, see e.g. \cite{fermi}, \cite{kondepudi}, is to derive from the Second law, by using the {\it Carnot cycle}, the inequality
\bq
\label{secondlaw}
\oint \frac{q(t)}{T(t)} dt \leq 0
\eq
for all cyclic processes, where equality holds for so-called {\it reversible} cyclic processes. Here $q$ denotes the heat flow supplied to the system and $T$ is the temperature. Clearly this is the same as cyclo-dissipativity with respect to the supply rate $-\frac{q}{T}$. Based on \eqref{secondlaw}, one defines the {\it entropy} $\cS$ as a function of the state of the thermodynamic system, and derives the {\it Clausius inequality} 
\bq
\label{clausius}
\cS(x(t_2))-\cS(x(t_1)) \geq \int_{t_1}^{t_2} \frac{q(t)}{T(t)} dt ,
\eq
with equality for reversible trajectories. This leads to the dissipativity formulation of the Second law as given in \cite{willems1972}, see also \cite{haddad}. Indeed, if the entropy $\cS$ is assumed to be {\it bounded from above}, then it follows that the thermodynamic system is {\it dissipative} with respect to the supply rate $-\frac{q}{T}$ and the nonnegative storage function $S=-\cS +c$, where $c$ is a suitable constant. However, for general thermodynamic systems there is no reason why the entropy $\cS$ is bounded from above; in which case $S=-\cS$ is truly indefinite, and the system is only cyclo-dissipative instead of dissipative.

Interesting, the usual definition \cite{fermi} of the entropy function $\cS$ satisfying \eqref{clausius} is based on considering {\it reversible} cycles (for which \eqref{secondlaw} holds with equality), and on the assumption that every state is reachable from a certain ground-state using such a reversible cycle. {\it Without} this assumption Theorem \ref{cycthm} comes into play. Indeed, assuming reachability from and controllability to a certain ground-state $x^*$ it follows from Theorem \ref{cycthm} that \eqref{secondlaw} implies the existence of storage functions $S$ satisfying $S_{ac}(x) \leq S(x) - S(x^*) \leq S_{rc}(x)$. Furthermore, for any such $S$ the function $\cS:=-S$ qualifies as a possible entropy function satisfying the Clausius inequality \eqref{clausius}. At the same time this leads to the question of uniqueness of the entropy function. For example, a 'maximal' entropy function could be defined as $\cS:=-S_{ac}$.

Clearly, in case all processes are reversible then $S_{ac}(x)=S_{rc}(x)$, leading to a unique (modulo a constant) storage and entropy function.

\subsection{One-port cyclo-passivity}
Cyclo-passivity may also naturally arise in the context of {\it one-port cyclo-passivity}; see \cite{submitted} for the general theory. 
An illustrative example is the {\it capacitor microphone}; see also \cite{submitted} for further details.
Consider an RC electrical circuit with voltage source $E$, where the capacitance $C(q)$ of the capacitor depends on the displacement $q$ of one of the plates of the capacitor, attached to a spring $k$ and a damper $d$, and affected by a mechanical force $F$ (air pressure).
The standard model is (in port-Hamiltonian formulation)
\bq
\begin{array}{rcl}
\left[\begin{array}{rcl}
\dot{q} \\
\dot{p} \\
\dot{Q}
\end{array}
\right] & = &
\left( \left[ 
\begin{array}{ccc}
0 & 1 & 0 \\ 
-1 & 0 & 0 \\ 
0 & 0 & 0
\end{array}
\right] -\left[ 
\begin{array}{ccc}
0 & 0 & 0 \\ 
0 & d & 0 \\ 
0 & 0 & 1/R
\end{array}
\right] \right)
\left[ 
\begin{array}{c}
\frac{\partial H}{\partial q} \\[2mm]
\frac{\partial H}{\partial p} \\[2mm]
\frac{\partial H}{\partial Q}
\end{array}
\right]\\
&& +\left[ 
\begin{array}{c}
0 \\ 
1 \\ 
0
\end{array}
\right] F+\left[ 
\begin{array}{c}
0 \\ 
0 \\ 
\frac{1}{R}\\
\end{array}
\right] E \\[6mm] 
v &= &\frac{\partial H}{\partial p}=\dot{q} \\[2mm]
I & = & \frac{1}{R}\frac{\partial H}{\partial Q}
\end{array}
\eq
where $Q$ is the charge at the capacitor, $p$ is the momentum of the movable plate with mass $m$, $R$ is the resistance of the resistor, and the Hamiltonian (total energy) equals
\bq
H(q,p,Q)=\frac{1}{2m}p^{2}+\frac{1}{2}kq^{2}+\frac{1}{2C(q)}
Q^{2}
\eq
It follows that
\bq
\frac{d}{dt}H= - d\dot{q}^2 - RI^2 + EI + Fv   \leq EI + Fv,
\eq
with $EI$ the {\it electrical power} and $Fv$ the \emph{mechanical power} supplied to the system. Thus this two-port system is {\it passive}.

For \emph{constant current} $I=\bar{I}$ at the electrical port it turns out \cite{submitted} that the system is cyclo-passive at the \emph{mechanical port}, with storage function (obtained by partial Legendre transformation of $H$)
\bq
\begin{array}{rcl}
H^*(q,p) & = & \frac{1}{2}kq^2 + \frac{p^2}{2m} + \frac{Q^2}{2C(q)} - \bar{V}Q =\\[2mm]
&=& \frac{1}{2}kq^2 + \frac{p^2}{2m} - \frac{1}{2}C^2(q)\bar{V}^2
\end{array}
\eq
where $\bar{V}= \frac{\partial H}{\partial Q}= R\bar{I}$. In fact
\bq
\frac{d}{dt}H^*= - d\dot{q}^2 - R\bar{I}^2 + Fv   \leq Fv
\eq
A typical expression for the capacitance $C(q)$ is $C(q) = \frac{1}{c_1 + c_2q}$, for certain constants $c_1,c_2\geq 0$. If $c_1 >0$ then
\bq
- \frac{1}{2}C^2(q)\bar{V}^2 \geq -\frac{1}{2c_1} \bar{V}^2,
\eq
and hence the system for $I=\bar{I}$ is \emph{passive} at the mechanical port. However, in case $c_1=0$ the storage function $H^*(q,p)$ for constant current $I=\bar{I} \neq 0$ is {\it not} bounded from below, and thus the system is only cyclo-passive at the mechanical port.

\section{Conclusions}
Cyclo-dissipativity and cyclo-passivity appears naturally in quite a few cases, from modeling to passivity-based control; see also \cite{willems2007} for additional motivation.
By symmetrizing and extending classical definitions and results of dissipativity theory we derived in a transparant way some new results concerning cyclo-dissipativity, including external characterizations and description of the set of (indefinite) storage functions.

\end{document}